\documentclass[11pt]{article}
\usepackage{amsmath}
\usepackage{amssymb, amscd, amsthm,amsfonts}
\usepackage[all]{xy}
\usepackage[dvips]{graphicx}
\usepackage{verbatim}
\usepackage[perpage,symbol*]{footmisc}
\newtheorem{theorem}{Theorem}
\newtheorem{lemma}{Lemma}

\newtheorem{corollary}{Corollary}
\newtheorem{conjecture}{Conjecture}

\begin{document}
	
	\pagestyle{myheadings}

	\title{\bf On the characterization of some algebraically defined bipartite graphs of girth 
eight\footnote{This work was supported by the Natural Science Foundation of China (No. 61977056).}}
	\author{ Ming Xu, Xiaoyan Cheng and Yuansheng Tang\footnote{Corresponding author.}\\
{\it\small School of Mathematical Sciences, Yangzhou University, Jiangsu, China\footnote{Email addresses: mx120170247@yzu.edu.cn(M. Xu), xycheng@yzu.edu.cn(X. Cheng), ystang@yzu.edu.cn(Y. Tang)}}
           }
	\date{}
	\maketitle

    {\noindent\small{\bf Abstract:}
For any field $\mathbb{F}$ and polynomials $f_{2},f_{3}\in\mathbb{F}[x,y]$,
let $\Gamma_{\mathbb{F}}(f_{2},f_{3})$ denote the bipartite graph with vertex partition $P\cup L$,
where $P$ and $L$ are two copies of $\mathbb{F}^{3}$,
and $(p_{1},p_{2},p_{3})\in P$ is adjacent to $[l_{1},l_{2},l_{3}]\in L$ if and only if
$p_{2}+l_{2}=f_{2}(p_{1},l_{1})$ and $p_{3}+l_{3}=f_{3}(p_{1},l_{1})$.
The graph $\Gamma_{3}(\mathbb{F})=\Gamma_{\mathbb{F}}(xy,xy^{2})$ is known to be of girth eight.
When $\mathbb{F}=\mathbb{F}_q$ is a finite field of odd size $q$ or $\mathbb{F}=\mathbb{F}_{\infty}$ is an algebraically closed field of characteristic zero, the graph $\Gamma_{3}(\mathbb{F})$ is conjectured to be the unique one with girth at least eight among those $\Gamma_{\mathbb{F}}(f_{2},f_{3})$ up to isomorphism.
This conjecture has been confirmed for the case that both $f_{2},f_{3}$ are monomials over $\mathbb{F}_q$, and for the case that at least one of $f_{2},f_{3}$ is a monomial over $\mathbb{F}_{\infty}$.
If one of $f_{2},f_{3}\in\mathbb{F}_q[x,y]$ is a monomial, it has also been proved the existence of a positive integer $M$ such that $G=\Gamma_{\mathbb{F}_{q^{M}}}(f_2,f_3)$ is isomorphic to $\Gamma_{3}(\mathbb{F}_{q^{M}})$ provided $G$ has girth at least eight.
In this paper, these results are shown to be valid when the restriction on the polynomials $f_2,f_3$ is relaxed further to that
one of them is the product of two univariate polynomials.
    Furthermore, all of such polynomials $f_2,f_3$ are characterized completely.}

    \vspace{1ex}

\section{Introduction}

All graphs considered in this paper are undirected, without loops and multiple edges. Let $G=(V,E)$ be a graph with vertex set $V=V(G)$ and edge set $E=E(G)$, where each edge in $E$ is a two-element subset of $V$. Two vertices $v,v'\in V$ are said adjacent to each other, and written
$v\sim v'$, if $\{v,v'\}\in E$ is an edge.
The order of $G$ is the number of vertices in $V$.
The degree of a vertex $v\in V$ is the number of vertices adjacent to $v$.
If every $v\in V$ has degree $t$, then $G$ is called a $t$-regular graph.
A sequence $(v_1,v_2,\ldots,v_k)\in V^{k}$ of vertices is called a $k$-$\ast$cycle of $G$ if $k\geq 3$, $v_1\sim v_2\sim \cdots\sim v_k\sim v_1$ and $v_i\neq v_{i+2}$, $i=1,2,\ldots,k$, where $v_{k+1}=v_1$, $v_{k+2}=v_2$.
A $k$-$\ast$cycle $(v_1,v_2,\ldots,v_k)$ of $G$ is called a $k$-cycle of $G$ if the vertices $v_1,v_2,\ldots,v_k$ differ from each other.
It is clear that any $k$-$\ast$cycle is a $k$-cycle if $3\leq k\leq 5$.
The graph $G$ is called bipartite if its vertex set can be divided into two parts as $V=P\cup L$ such that
each edge in $E$ consists of a vertex in $P$ and a vertex in $L$.
It is not difficult to show that, in a bipartite graph, there is no $(2k+1)$-$\ast$cycle, any $6$-$\ast$cycle is a $6$-cycle and any $8$-$\ast$cycle is either an $8$-cycle or the concatenation of two $4$-cycles.
In particular, in a bipartite graph with no 4-cycle, any $8$-$\ast$cycle is an $8$-cycle.
If $G$ has some cycles, its girth is defined as the largest integer $k$ such that $G$ contains no $i$-cycles for any $k$ with $3\leq i<k$. Other standard graph theory definitions can be found in \cite{Bollobas98}.

\indent For $ k\geqslant 2 $, let $ g_{k}(n) $ be the greatest number of edges in a graph of order $ n $ with girth at least $ 2k+1 $. It is well known that for sufficiently large $ n $,
$$c_{k}'n^{1+\frac{2}{3k-3+\epsilon}} \leqslant g_{k}(n)\leqslant c_{k}n^{1+\frac{1}{k}}, $$
where $ c_{k}' $ and $ c_{k} $ are positive constants depending only on $ k $, and $\epsilon=0$ if $ k $ is odd and $ \epsilon=1 $ if $ k $ is even (see \cite{Hou17}). The upper bound comes from \cite{Bondy74} and the lower bound from an explicit construction of \cite{Lazebnik95}. The upper bound is known to be sharp in magnitude $ n^{1+\frac{1}{k}} $ only for $ k=2,3,5 $. In this paper, we concentrate on the case of $ k=3 $. A known example which provides such extremal magnitude is the bipartite graph $\Gamma_{3}(\mathbb{F}_{q})$ with vertex partition
$P\cup L$, where $\mathbb{F}_{q}$ is the finite field of $q$ elements, $P$ and $L$ are two copies of $\mathbb{F}^{3}_q$, and $(p_{1},p_{2},p_{3})\in P$ is adjacent to $[l_{1},l_{2},l_{3}]\in L$ if and only if
$p_{2}+l_{2}=p_{1}l_{1}$ and $p_{3}+l_{3}=p^2_{1}l_{1}$.
It is can be shown easily that $\Gamma_{3}(\mathbb{F}_{q})$ is a $ q $-regular bipartite graph of order $ 2q^{3} $ and girth eight.
When $ q$ is an odd prime power, the graph $ \Gamma_{3}(\mathbb{F}_{q})$ is isomorphic to an induced subgraph of the point-line incidence graph of the classical generalized quadrangle $W(q)$ of order $q$ (see \cite{Dmytrenko04,Dmytrenko07,Kronenthal12,Kronenthal16,Kronenthal19}).

From now on, we focus on a generalization of the graph $ \Gamma_{3}(\mathbb{F}_q) $. Let $ \mathbb{F} $ be an arbitrary field, for polynomials $ f_{2},f_{3}\in \mathbb{F}[x,y] $, the graph $ \Gamma_{\mathbb{F}}(f_{2},f_{3}) $ is a bipartite graph with vertex partition $ P\cup L $, where $P$ and $L$ are two copies of $\mathbb{F}^{3}$, and $(p_{1},p_{2},p_{3})\in P$ is adjacent to $[l_{1},l_{2},l_{3}]\in L$ if and only if
$$ p_{2}+l_{2}=f_{2}(p_{1},l_{1})\text{ and } p_{3}+l_{3}=f_{3}(p_{1},l_{1}). $$
When $\mathbb{F}=\mathbb{F}_{q}$, we simplify the notation $ \Gamma_{\mathbb{F}_q}(f_{2},f_{3}) $ to $ \Gamma_{q}(f_{2},f_{3}) $ further.
For odd prime power $q$, it is of interest to find a graph $ \Gamma_{q}(f_{2},f_{3}) $ of girth eight that is not isomorphic to $ \Gamma_{3}(\mathbb{F}_{q}) $, since a new generalized quadrangle could be constructed by ``attaching'' some tree to such graph.
However, many attempts towards this aim failed (see \cite{Dmytrenko04,Dmytrenko07,Kronenthal12, Kronenthal16, Kronenthal19}). On the contrary, the following uniqueness conjecture was proposed in
\cite{Kronenthal16, Kronenthal19}:
\begin{conjecture}
\label{conj01}	If $\mathbb{F}=\mathbb{F}_q$ is a finite field of odd size or $\mathbb{F}=\mathbb{F}_{\infty}$ is an algebraically closed field of characteristic zero, then every graph $\Gamma_{\mathbb{F}}(f_{2},f_{3})$ of girth at least eight is isomorphic to
    $\Gamma_{3}(\mathbb{F})=\Gamma_{\mathbb{F}}(xy,x^{2}y)$.
\end{conjecture}

When $\mathbb{F}=\mathbb{F}_q$ is a finite field of odd size and $f_2, f_3$ are monomials, Conjecture~\ref{conj01} was investigated in
\cite{Dmytrenko04,Dmytrenko07,Kronenthal12} and confirmed in \cite{Hou17}.
When $\mathbb{F}=\mathbb{F}_{\infty}$ is an algebraically closed field of characteristic zero and at least one of $f_2, f_3$ is monomial, Conjecture~\ref{conj01} was confirmed in \cite{Kronenthal16,Kronenthal19}.
When $\mathbb{F}=\mathbb{F}_q$ is a finite field of odd size and one of $f_2, f_3$ is monomial,
the following result was also shown in \cite{Kronenthal19}:
If $q$ is a power of some odd prime $p$ and $f\in \mathbb{F}_{q}[x,y] $ has degree at most $p-2$ with respect to each of $x$ and $y$, then, for any integers $k$, $m$ coprime to p, there exists a positive integer $M=M(k,m,q)$ such that, for all positive integers $r$, every graph $\Gamma_{{q^{Mr}}}(x^{k}y^{m},f)$ of girth at least eight is isomorphic to $\Gamma_{3}(\mathbb{F}_{q^{Mr}})$.
In this paper, we will show that the main results of \cite{Kronenthal19} are still true if one of $f_2,f_3$ is of form $f(x)g(y)$ for univariate polynomials $f,g\in\mathbb{F}[x]$.

For any field $ \mathbb{F} $, let $\mathbb{F}^{*}=\mathbb{F}\backslash \{0\}$. For positive integer $k$, let $[1,k]$ denote the set $\{1,2,\cdots,k\} $, $\mathbb{F}[x]_{k}$ the set of polynomials in $\mathbb{F}[x]$ of degree at most $k$
and $\mathbb{F}[x,y]_{k}$ the set of polynomials in $\mathbb{F}[x,y]$ of degree at most $k$ with respect to each of $x$ and $y$, respectively.
Through this paper, we assume that $q$ is a power of some prime $p$ and $m,n$ are positive integers such that
\begin{align}\label{eq1}
q > \max\{2mn+3,mn+3n+1,n(n+1)+2\}.
\end{align}
Let $M=M(mn)$ be the least common multiple of the integers $ 2,3,\cdots ,mn $. Clearly, any polynomials $ T(x)\in \mathbb{F}_{q}[x]_{mn} $ can be decomposed completely in $\mathbb{F}_{q^{M}}$. For any $a\in \mathbb{F}_{q}$, let $\rho_{a}(x)=x(x-a)\in\mathbb{F}_{q}[x]$. Let $K_{p}=\{p^{j}\mid j\geqslant 0\}$. For $u,v\in K_{p}$, let $\Phi_{p}(u,v)=\{(i,j)\in K_{p}^{2}:iv=ju\}$.
Let $ f,g\in \mathbb{F}_{q}[x]_{m} $ be monic polynomials with $ f(0)=g(0)=0 $ and 
\begin{align}
h(x,y)=\sum_{1\leq i,j\leq n} h_{i,j}x^{i}y^{j}\in \mathbb{F}_{q}[x,y]_n\label{eq1'}
\end{align}
be a nonzero polynomial. For $a\in\mathbb{F}_q$, $u,v\in K_p$ and the polynomial $h$ given in (\ref{eq1'}), let
$\mu_{a,u,v}(h)$, $\nu_{a,u,v}(h)$ and $\pi_{u,v}(h)$ denote the polynomials in $\mathbb{F}_q[x,y]$ defined respectively by
\begin{gather*}
    \mu_{a,u,v}(h)(x,y)=h(x,y)-\sum_{(i,j)\in\Phi_{p}(u,v)}h_{2i,j}\rho_{a}^{i}(x)y^{j},\\
    \nu_{a,u,v}(h)(x,y)=h(x,y)-\sum_{(i,j)\in\Phi_{p}(v,u)}h_{i,2j}x^{i}\rho_{a}^{j}(y),\\
    \pi_{u,v}(h)(x,y)=h(x,y)-\sum_{(i,j)\in\Phi_{p}(u,v)}h_{i,j}x^{i}(x)y^{j}.
\end{gather*}

The main result of this paper is as follows.

\begin{theorem}\label{one}
The graph $G=\Gamma_{{q^{M}}}(f(x)g(y),h(x,y))$ is isomorphic to $\Gamma_{3}(\mathbb{F}_{q^{M}})$ if it has girth at least eight. 
Furthermore, $G$ has girth at least eight
if and only if there are some $ a\in \mathbb{F}_{q} $, $ \zeta\in\mathbb{F}_{q}^{*} $, $ u,v\in K_{p}\cap[1,m] $ and $s\in K_{p}\cap[1,n]$ such that one of the following is valid.
\begin{enumerate}
  \item[(i)] $ f(x)=\rho_{a}^{u}(x), g(y)=y^{v} $ and $\mu_{a,u,v}(h)(x,y)=\zeta x^{su/v}y^{s}$.
  \item[(ii)] $ f(x)=x^{v}, g(y)=\rho_{a}^{u}(y) $ and $\nu_{a,u,v}(h)(x,y)=\zeta x^{s}y^{su/v}$.
  \item[(iii)] $ f(x)=x^{u}, g(y)=y^{v} $ and $\pi_{u,v}(h)(x,y)=\zeta x^{2su/v}y^{s}$ or $\zeta x^{s}y^{2sv/u}$.
  \item[(iv)] $p=2$, $a\neq 0$ and either
  \begin{enumerate}
    \item $f(x)=\rho_a^u(x)$, $g(y)=y^{2su}$ and $\mu_{a,u,2su}(h)(x,y)=\zeta xy^{s}$, or
    \item $f(x)=x^{2su}$, $g(y)=\rho_a^u(y)$ and $\nu_{a,u,2su}(h)(x,y)=\zeta x^{s}y$.
  \end{enumerate}
\end{enumerate}
\end{theorem}

The following theorem is an analog of Theorem \ref{one} for the case that $\mathbb{F}= \mathbb{F}_{\infty}$ is an algebraically closed field of characteristic zero.

\begin{theorem}\label{two}
	Suppose that $ f,g\in \mathbb{F}_{\infty}[x] $ are monic polynomials with $ f(0)=g(0)=0 $ and $ h(x,y)=\sum_{ i,j\geqslant 1} h_{i,j}x^{i}y^{j}\in \mathbb{F}_{\infty}[x,y] $ is a nonzero polynomial. The graph $G= \Gamma_{\mathbb{F}_{\infty}}(f(x)g(y),h(x,y)) $ 
is isomorphic to $\Gamma_{3}(\mathbb{F}_{\infty})$ if $G$ has girth at least eight. 
Furthermore, $G$ has girth at least eight if and only if there are some  $ a\in \mathbb{F}_{\infty} $, $ \zeta\in\mathbb{F}_{\infty}^{*} $ such that one of the following is valid.
\begin{enumerate}
  \item[(i)] $ f(x)=\rho_{a}(x), g(y)=y $ and $ h(x,y)=\zeta xy+h_{2,1}\rho_{a}(x)y. $
  \item[(ii)] $ f(x)=x, g(y)=\rho_{a}(y) $ and $ h(x,y)=\zeta xy+h_{1,2}x\rho_{a}(y). $
  \item[(iii)] $ f(x)=x, g(y)=y $ and $ h(x,y)=\zeta\rho_{a}(x)y$ or $\zeta x\rho_{a}(y)$.
\end{enumerate}
\end{theorem}

This paper is organized as follows. 
In Section 2 we show some preliminaries, including a necessary and sufficient condition of $2k$-*cycles in $\Gamma_{\mathbb{F}}(f_{2},f_{3})$, 
some isomorphisms of $\Gamma_{\mathbb{F}}(f_{2},f_{3})$, some simple conclusions on a few monomial graphs, and a simple but useful lemma on
the characterization of polynomials according to their roots in some extension field.
In Sections 3 to 5, we characterize the monic polynomials $f,g\in\mathbb{F}_q[x]_m$ and nonzero polynomial $h\in\mathbb{F}_q[x,y]_n$
under the conditions $f(0)=g(0)=h(0,y)=h(x,0)=0$ and that $G=\Gamma_{q^M}(f(x)g(y),h(x,y))$ has girth at least eight.
In Section 6, we complete the proof of Theorem~\ref{one}, and 
make some concluding remarks, including a simple illustration for the proof of Theorem\ref{two}.

\section{Preliminaries}

For $k\geq 2$, let $\Delta_k$ be the function defined by
\begin{align*}
\Delta_k :\mathbb{F}[x, y]&\rightarrow\mathbb{F}[x_1,\ldots, x_k; y_1,\ldots, y_k],\\
f(x, y)&\mapsto \sum_{i=1}^k(f(x_i, y_i)-f(x_{i+1}, y_i)),
\end{align*}
where $x_{k+1}=x_1$.
Suppose $k\geq 2$, $S=(a_1,\ldots,a_k;r_1,\ldots,r_k)\in \mathbb{F}^{2k}$ and $f_2,f_3\in\mathbb{F}[x,y]$,
it is clear that the graph $G=\Gamma_{\mathbb{F}}(f_{2},f_{3})$ contains a $2k$-$\ast$cycle of form
\begin{align}\label{0001}
((a_{1},b_{1},c_{1}),[r_{1},s_{1},t_{1}], \ldots,(a_{k},b_{k},c_{k}),[r_{k},s_{k},t_{k}])
\end{align}
if and only if
	\begin{equation}
	\begin{cases}
	\Delta_{k}(f_{2})(S)=\Delta_{k}(f_{3})(S)=0,\\
	a_i\neq a_{i+1},\,r_i\neq r_{i+1}, \,i\in[1,k],
	\end{cases}
	\end{equation}
where $a_{k+1}=a_1$ and $r_{k+1}=r_1$ (see \cite{Dmytrenko07}). If $G$ contains some $2k$-$\ast$cycles of form (\ref{0001}),
we also call $S=(a_1,\ldots,a_k;r_1,\ldots,r_k)$ a $2k$-$\ast$cycle of $G$ for simplicity.

Some useful isomorphisms of the graphs $\Gamma_{\mathbb{F}}(f_{2}, f_{3})$ are integrated in the following lemma.

\begin{lemma}\label{iso}
	Assume $ f_{2}, f_{3}\in \mathbb{F}[x,y] $. Then
\begin{enumerate}
  \item[(i)] $\Gamma_{\mathbb{F}}(f_{2}, f_{3})\cong \Gamma_{\mathbb{F}}(f_{3}, f_{2})$.
  \item[(ii)] $\Gamma_{\mathbb{F}}(f_{2}, f_{3})\cong \Gamma_{\mathbb{F}}(\bar{f}_{2}, \bar{f}_{3})$, where $\bar{f}(x,y)=f(y,x)$.
  \item[(iii)] For any $ \alpha\in \mathbb{F^{*}},\,\Gamma_{\mathbb{F}}(f_{2}, f_{3})\cong \Gamma_{\mathbb{F}}(f_{2}, \alpha f_{3})$.
  \item[(iv)] For any $ \beta\in \mathbb{F},\,\Gamma_{\mathbb{F}}(f_{2}, f_{3})\cong \Gamma_{\mathbb{F}}(f_{2}, f_{3}+\beta f_{2}) $.
  \item[(v)] For any $ t\in \mathbb{F}[x], \,\Gamma_{\mathbb{F}}(f_{2}, f_{3})\cong \Gamma_{\mathbb{F}}(f_{2}(x,y), f_{3}(x,y)+t(x)) $.
\end{enumerate}
Furthermore, if $ \mathbb{F}=\mathbb{F}_{q} $, then, for any $ u\in K_{p} $,
\begin{enumerate}
  \item[(vi)] $\Gamma_{{q}}(f_{2}, f_{3})\cong \Gamma_{{q}}(f_{2}(x^{u},y), f_{3}(x^{u},y))$.
  \item[(vii)] $\Gamma_{{q}}(f_{2}, f_{3})\cong \Gamma_{{q}}(f_{2}^{u}, f_{3})$.
\end{enumerate}
\end{lemma}

\begin{proof}
	We refer the reader to \cite{Kronenthal16} for the proofs of $ (i)\sim (v) $.

For $ (vi) $, let $v$ be the least positive integer such that $vu$ is a power of $q$ and $\pi_1$ the map from $V(\Gamma_{{q}}(f_{2}, f_{3}))$ to $V(\Gamma_{{q}}(f_{2}(x^{u},y), f_{3}(x^{u},y)))$
defined by $ \pi_1 :(p_{1},p_{2},p_{3})\mapsto (p_{1}^{v},p_{2},p_{3}) $ and $ \pi_1 :[l_{1},l_{2},l_{3}]\mapsto [l_{1},l_{2},l_{3}]$.
Then, $\pi_1$ is a graph isomorphism.

For $ (vii) $, let $\pi_2$ be the map from $V(\Gamma_{{q}}(f_{2}, f_{3}))$ to $V(\Gamma_{{q}}(f^{u}_{2},$ $ f_{3}))$
defined by $ \pi_2 :(p_{1},p_{2},p_{3})\mapsto (p_{1},p_{2}^{u},p_{3}) $ and $ \pi_2 :[l_{1},l_{2},l_{3}]\mapsto [l_{1},l_{2}^{u},l_{3}]$. Then,  $\pi_2$ is also a graph isomorphism.
\end{proof}

If $f_2,f_3$ are monomials, the graph $\Gamma_{\mathbb{F}}(f_2,f_3)$ is referred to as a monomial graph.
Now we show some simple results for a few monomial graphs.
\begin{lemma}\label{2}
	$ (i) $ The girth of $ \Gamma_{3}(\mathbb{F})=\Gamma_{\mathbb{F}}(xy,x^{2}y) $ is $ 8 $.
	\\$ (ii) $ The girth of $ \Gamma_{\mathbb{F}}(x^{3}y,x^{2}y) $ is 6 if $\mathbb{F}\not\in\{\mathbb{F}_2,\mathbb{F}_3,\mathbb{F}_5\}$.
	\\$ (iii) $ The girth of $ \Gamma_{\mathbb{F}}(xy,x^{2}y^{3}) $ is 6 if $\mathbb{F}\neq \mathbb{F}_3$ and the characterstic of $\mathbb{F}$ is not equal to 2.
	\\$ (iv) $ $\Gamma_{{5}}(x^{3}y,x^{2}y)\cong \Gamma_{3}(\mathbb{F}_{5})$.
	\\$ (v) $ $\Gamma_{{3}}(x^{3}y,x^{2}y)\cong \Gamma_{{3}}(xy,x^{2}y^{3})\cong \Gamma_{{3}}(\mathbb{F}_3)$.
\end{lemma}

\begin{proof}
	Let $(a,b)$ and $(c,d)$ be two pairs of distinct elements in $\mathbb{F}$. From $ \Delta_{2}(xy)(a,b;c,d)$ $=(a-b)(c-d)\neq 0 $, we see that $ \Gamma_{3}(\mathbb{F}) $ and $ \Gamma_{\mathbb{F}}(xy,x^{2}y^{3}) $ contain no 4-cycle and thus have girth at least 6.
	
	\indent $ (i) $ Let $ S_{0}=(a,b,e;c,d,f) $, where $ e\in \mathbb{F}\backslash\{ a,b \} $ and $ f\in \mathbb{F}\backslash\{ c,d \} $. If $ \Delta_{3}(xy)(S_{0})=(a-b)(c-d)-(a-e)(f-d)$ is equal to 0, then we have
\begin{align*}
\Delta_{3}(x^{2}y)(S_{0})=&(a^{2}-b^{2})(c-d)-(a^{2}-e^{2})(f-d)\\
=&(a-b)(c-d)(e-b)\neq 0.
\end{align*}
Hence, $ \Gamma_{3}(\mathbb{F}) $ contains no 6-cycle. Furthermore, $ \Gamma_{3}(\mathbb{F}) $ has girth 8 since it contains the 8-cycle $ (1,0,1,0;1,0,-1,0) $.
	
	\indent $ (ii) $ If $ \Delta_{2}(x^{2}y)(a,b;c,d)=(a^{2}-b^{2})(c-d)$ is equal to 0, then we have $ a^{2}=b^{2}\neq 0 $ and thus
	$$ \Delta_{2}(x^{3}y)(a,b;c,d)=(a^{3}-b^{3})(c-d)=a^{2}(a-b)(c-d)\neq 0. $$
	Hence, $ \Gamma_{\mathbb{F}}(x^{3}y,x^{2}y) $ has no 4-cycle and has girth at least 6. Furthermore, if there is some $t\in \mathbb{F}\setminus\{0,1,-1\}$ such that $0\not\in\{t-2,2t-1\}$, then
$$(t,1-t,t(t-1);t^{2}(t-1)^{2},t^{2},(t-1)^{2})$$ is a 6-cycle of $ \Gamma_{\mathbb{F}}(x^{3}y,x^{2}y) $.
Hence, the girth of $ \Gamma_{\mathbb{F}}(x^{3}y,x^{2}y) $ is equal to 6 if $\mathbb{F}\not\in\{\mathbb{F}_2,\mathbb{F}_3,\mathbb{F}_5\}$.

	\indent $ (iii) $ If $\mathbb{F}\neq \mathbb{F}_3$ and the characterstic of $\mathbb{F}$ is not equal to 2, for any $ t\in \mathbb{F}\backslash \{0,1,-1\}$, $$(-t,t+2t^{2},t+2;1,0,t)$$ is a 6-cycle of $ \Gamma_{\mathbb{F}}(xy,x^{2}y^{3}) $ and thus $\Gamma_{\mathbb{F}}(xy,x^{2}y^{3})$ has girth 6.
	
	\indent $ (iv) $ Since $ x^{3} $ is a permutation in $ \mathbb{F}_{5} $, from $ (x^{3})^{2}\equiv x^{2}(\bmod\,x^{5}-x) $ we have $$ \Gamma_{{5}}(x^{3}y,x^{2}y)\cong \Gamma_{{5}}(x^{3}y,(x^{3})^{2}y)\cong \Gamma_{{5}}(xy,x^{2}y)=\Gamma_3(\mathbb{F}_5).$$
	
	\indent $ (v) $ The desired proof follows simply from $ x^{3}\equiv x(\bmod\,x^{3}-x) $.
\end{proof}

To deal with the graph $\Gamma_{\mathbb{F}}(f_{2},f_{3})$ in general,
according to Lemma \ref{iso} one can assume, without loss of generality, that $f_{2}$ and $f_{3}$ consist of only mixed terms, i.e. $ f_{i}(x,0) $ and $ f_{i}(0,y) $ are zero polynomials for $i=2,3$. Hereafter, we assume that the bipartite graph $G=\Gamma_{{q^{M}}}(f(x)g(y),h(x,y))$ has girth at least 8, where $ f,g\in \mathbb{F}_{q}[x]_{m} $ are monic polynomials with $ f(0)=g(0)=0 $ and $ h(x,y)=\sum_{1\leq i,j\leq n} h_{i,j}x^{i}y^{j}\in \mathbb{F}_{q}[x,y]_n $ is a nonzero polynomial.

In the end of this section we show a lemma which is useful for the characterization of the polynomials $f,g,h$.
\begin{lemma}\label{7}
Suppose that $1\leq D\leq mn$, $1\leq N<q/2$ and $W\subseteq \mathbb{F}_{q}^{2}$ is a nonempty set such that, for any $(a,b)\in W$,
	\begin{equation}\label{eq7}
	  \min\{\lvert \{c\in \mathbb{F}_{q}:(a,c)\in W\}\rvert , \lvert\{d\in \mathbb{F}_{q}:(d,b)\in W\}\rvert\}>2N.
	\end{equation}
Let $\{e_{i}\}_{1\leq i\leq D}$ be a family of polynomials in $\mathbb{F}_{q}[x,y]_N$ such that, for any $ (a,b)\in W $, the $ t $-polynomial $ \sum_{1\leqslant i\leqslant D}e_{i}(a,b)t^{i}\in \mathbb{F}_{q}[t]_{D} $ has no root in $ \mathbb{F}_{q^{M}}^{*} $. Then, there is an integer $ s\in [1,D] $ such that $e_{s}(a,b)\neq 0$ for each $(a,b)\in W$ and $e_{i}$ is the zero polynomial for any $ i\neq s $.
\end{lemma}

\begin{proof}
	Let $ s\in [1,D] $ be an integer such that $ f_{s}$ is not the zero polynomial. Let $e_s(x,y)=\sum_{0\leq i,j\leq N}w_{i,j}x^iy^j$.
If $e_s$ is always equal to 0 over $W$, then, for any $(a,b)\in W$, from $\lvert \{c\in \mathbb{F}_{q}:(a,c)\in W\}\rvert>2N$ we see $\sum_{0\leq i\leq N}w_{i,j}a^i=0$, $0\leq j\leq N$ and thus from $\lvert\{d\in \mathbb{F}_{q}:(d,b)\in W\}\rvert>2N$ we see $w_{i,j}=0$, $0\leq i,j\leq N$, contradicts to the assumption.
Hence, there is some pair
$(a,b)\in W$ such that $e_{s}(a,b)\neq 0$.
For such a pair $(a,b)$, let
	$$ A=\{d\in \mathbb{F}_{q}:(d,b)\in W,e_{s}(d,b)\neq 0\}, $$
	$$ B=\{c\in \mathbb{F}_{q}:(a,c)\in W,e_{s}(a,c)\neq 0\}. $$
	From (\ref{eq7}) and that the polynomials $e_{s}(a,y)\in \mathbb{F}_{q}[y]_N $ and $e_{s}(x,b)\in \mathbb{F}_{q}[x]_N $ have degree at most $ N $, we see
	\begin{equation}\label{eq8}
	  \min\{\lvert A\rvert ,\lvert B \lvert\}>N.
	\end{equation}
Since for any $ d\in A $ the $ t $-polynomial $ \sum_{1\leqslant i\leqslant D}e_{i}(d,b)t^{i}\in \mathbb{F}_{q}[t]_{D} $ can be decomposed completely in $\mathbb{F}_{q^{M}}$ and has no nonzero root, we have $e_{i}(d,b)=0 $ for any $ i\neq s $. Hence, from (\ref{eq8}) the $ x $-polynomial $e_{i}(x,b)\in \mathbb{F}_{q}[x]_{N} $ is the zero polynomial for any $i\neq s $. Similarly, one can conclude that the $ y $-polynomial $e_{i}(a,y)\in \mathbb{F}_{q}[y]_{N} $ is the zero polynomial for any $ i\neq s $. Therefore, from (\ref{eq8}) we see that the polynomial $e_{i}(x,y)\in \mathbb{F}_{q}[x,y] $ is the zero polynomial for any $ i\neq s $. Clearly, we have $e_{s}(a,b)\neq 0 $ for any $ (a,b)\in W $.
\end{proof}

\section{Characterization of $f,g$}

In this section, we consider to characterize the univariate polynomials $f,g$.

\begin{lemma}\label{3}
$ (i) $ If $ a,b\in \mathbb{F}_{q^{M}} $ are distinct with $ f(a)=f(b) $, then the $y$-polynomial
$$ \theta_{a,b}(y)=h(a,y)-h(b,y)\in \mathbb{F}_{q^{M}}[y]  $$
is injective in $ \mathbb{F}_{q^{M}} $.
\\ $ (ii) $ If $ c,d\in \mathbb{F}_{q^{M}} $ are distinct with $ g(c)=g(d) $, then the $x$-polynomial
$$ \phi_{c,d}(x)=h(x,c)-h(x,d)\in \mathbb{F}_{q^{M}}[x]  $$
is injective in $ \mathbb{F}_{q^{M}} $.
\end{lemma}

\begin{proof}
Since $ (ii) $ is symmetrical to $ (i) $, we only give proof for $ (i) $.

Suppose $ a,b\in \mathbb{F}_{q^{M}} $ are distinct with $f(a)=f(b)$. Let $ S_{1}=(a,b;c,d) $. Since $ G $ has no 4-cycle, from $ \Delta_{2}(f(x)g(y))(S_{1})=0 $ we see
	$$ \Delta_{2}(h(x,y))(S_{1})=\theta_{a,b}(c)-\theta_{a,b}(d)\neq 0 $$
	for any $ c,d \in \mathbb{F}_{q^{M}} $ with $ c\neq d $, and thus $ \theta_{a,b}(y) $ must be injective in $ \mathbb{F}_{q^{M}} $.
\end{proof}

\begin{lemma}\label{4}
	At least one of the polynomials $ f, g$ is injective in $ \mathbb{F}_{q^{M}} $.
\end{lemma}

\begin{proof}
Suppose that neither $ f$ nor $ g$ is injective in $ \mathbb{F}_{q^{M}} $. Let $ a,b,c,d $ be elements in $ \mathbb{F}_{q^{M}}$ with $a\neq b$, $c\neq d$
such that $f(a)=f(b)$, $ g(c)=g(d) $. According to Lemma \ref{3}, the polynomials $ \theta_{a,b},\phi_{c,d} \in\mathbb{F}_{q^{M}}[x] $ are injective in $ \mathbb{F}_{q^{M}} $. For $ S_{2}=(a,b,t';t,c,d) $, we have
$\Delta_{3}(f(x)g(y))(S_{2})=0$
and
$$ \Delta_{3}(h(x,y))(S_{2})=h(b,c)-h(a,d)+\theta_{a,b}(t)-\phi_{c,d}(t').$$
Since $ \theta_{a,b}$ is injective in $ \mathbb{F}_{q^{M}} $, for any $ t\in \mathbb{F}_{q^{M}}\backslash \{c,d\} $ we have
$$ h(b,c)-h(a,d)+\theta_{a,b}(t)\notin \{\phi_{c,d}(a),\phi_{c,d}(b)\}.$$
Therefore, from that $\phi_{c,d}$ is injective in $\mathbb{F}_{q^{M}}$, for $t\in \mathbb{F}_{q^{M}}\backslash \{c,d\}$ there exists some $t'\in \mathbb{F}_{q^{M}}\backslash\{a,b\}$ such that $ \Delta_{3}(h(x,y))(S_{2})=0 $ and thus $ G $ has a 6-cycle of form $ S_{2} $, contradicts to the assumption.
\end{proof}
\begin{lemma}\label{5}
There are no distinct $x_{0},x_{1},x_{2}\in \mathbb{F}_{q^{M}}$ satisfying $f(x_{0})=f(x_{1})=f(x_{2})$ or $ g(x_{0})=g(x_{1})=g(x_{2})$.
\end{lemma}

\begin{proof}
	Assume in contrast that distinct $ x_{0},x_{1},x_{2}\in \mathbb{F}_{q^{M}} $ satisfy $ f(x_{0})=f(x_{1})=f(x_{2}) $.
For $ S_{3}=(x_{0},x_{1},x_{2};y_{0},y_{1},y_{2}) $, we have $ \Delta_{3}(f(x)g(y))(S_{3})=0 $ and
	$$ \Delta_{3}(h(x,y))(S_{3})=\theta_{x_{0},x_{1}}(y_{0})+\theta_{x_{1},x_{2}}(y_{1})+\theta_{x_{2},x_{0}}(y_{2}),$$
where the polynomials $\theta_{x_{0},x_{1}},\theta_{x_{1},x_{2}},\theta_{x_{2},x_{0}}$ are injective in $\mathbb{F}_{q^{M}}$ according to Lemma~\ref{3}.
Therefore, from $\theta_{x_{0},x_{1}}+\theta_{x_{1},x_{2}}+\theta_{x_{2},x_{0}}=0$ there are distinct $ y_{0},y_{1},y_{2}\in \mathbb{F}_{q^{M}} $ satisfying $ \Delta_{3}(h(x,y))(S_{3})=0 $, and thus $ G $ contains a 6-cycle of form $ S_{3} $, contradicts to the assumption.\\
	\indent Similarly, one can show that there are no distinct $ x_{0},x_{1},x_{2}\in \mathbb{F}_{q^{M}} $ satisfying $ g(x_{0})=g(x_{1})=g(x_{2}) $.
\end{proof}

\begin{lemma}\label{6}
	Let $D\in[1,mn]$ and $ T(x)\in \mathbb{F}_{q}[x]_{D} $ be a monic polynomial with $ T(0)=0 $.\\
	\noindent $(i)$ If $ T(x) $ is injective in $ \mathbb{F}_{q^{M}} $, then there is some $ u\in K_{p} $ such that	
	\begin{equation}\label{eq2}
      T(x)=x^{u}.
	\end{equation}	
	\noindent $(ii)$ If $ T(x) $ is not injective in $ \mathbb{F}_{q^{M}} $ and there are no distinct $ x_{0},x_{1},x_{2}\in \mathbb{F}_{q^{M}} $ with $ T(x_{0})=T(x_{1})=T(x_{2}) $, then there are some $ v\in K_{p} $ and $ a\in \mathbb{F}_{q} $ such that
	\begin{equation}\label{eq3}
	  T(x)=\rho_a^v(x).
	\end{equation}	
\end{lemma}

\begin{proof}
	From $D\in[1, mn]$ we see any polynomial in $\mathbb{F}_{q}[x]_{D}$ can be completely decomposed in $ \mathbb{F}_{q^{M}}^{*} $.

Assume first that the polynomial $T(x)\in\mathbb{F}_{q}[x]_{D} $ has no root in $ \mathbb{F}_{q^{M}}^{*} $. Clearly, we have $ T(x)=x^{d} $, where $d$ is the degree of $T(x)$. Let $ u $ be the largest integer in $ K_{p} $ with $u\mid d$. Then the integer $k=d/u\in[1,D]$ is coprime to $p$. We note that $ x^{k}-1 $ has no repeated roots in $ \mathbb{F}_{q^{M}}^{*} $ since its derivative $ kx^{k-1} $ has no root in $ \mathbb{F}_{q^{M}}^{*} $. Therefore, the polynomial $ x^{k}-1\in\mathbb{F}_q[x]_N $ has $ k $ distinct roots in $ \mathbb{F}_{q^{M}}^{*} $. If $ T(x) $ is injective in $ \mathbb{F}_{q^{M}} $, we have $ k=1 $ and thus (\ref{eq2}) follows, where we note that $ x^{i} $ is an injective in $ \mathbb{F}_{q^{M}} $ if and only if $i$ is coprime to $q^M-1$. If $ T(x) $ is not injective in $ \mathbb{F}_{q^{M}} $ and there are no distinct $ x_{0},x_{1},x_{2}\in \mathbb{F}_{q^{M}} $ with $ T(x_{0})=T(x_{1})=T(x_{2}) $, then we must have $ k=2 $ and thus (\ref{eq3}) is valid for $ a=0 $.
	
	\indent Assume now that  $ T(x)\in \mathbb{F}_{q}[x]_{D} $ has roots in $ \mathbb{F}_{q^{M}}^{*} $ and there are no distinct $ x_{0},x_{1},x_{2}\in \mathbb{F}_{q^{M}} $ with $ T(x_{0})=T(x_{1})=T(x_{2}) $. Clearly, $ T(x) $ has just one root in $ \mathbb{F}_{q^{M}}^{*} $. Hence, we have $ T(x)=x^{r}(x-a)^{s} $ for some $ a\in \mathbb{F}_{q^{M}}^{*} $ and positive integers $ r,s $. Let $ v$ be the largest integer in $K_p$ dividing both $ r $ and $ s $. Then, the positive integers $ k=r/v $ and $ l=s/v $ satisfy $ 2\leqslant k+l \leqslant D$ and that at least one of $k+l$ and $k$ is not divided by $p$.
Hence, the $ x $-polynomial $ (k+l)x-ka\in \mathbb{F}_{q^{M}}[x] $ is not the zero polynomial.
	
	\indent Furthermore, we assume $ k+l>2 $ and write $ R(x)=x^{k}(x-a)^{l} $. From $v\in K_p$ and $ R^{v}(x)=T(x)$, we see that $R(x)$ is also a polynomial in $\mathbb{F}_q[x]_{D}$ and there are no distinct $ x_{0},x_{1},x_{2}\in \mathbb{F}_{q^{M}} $ with $ R(x_{0})=R(x_{1})=R(x_{2}) $. Therefore, for any $ b\in \mathbb{F}_{q}\backslash \{0,a\} $ we have $ R(b)\in  \mathbb{F}_{q}^{*} $ and there are some $ \alpha(b)\in \mathbb{F}_{q^{M}}\backslash \{0,a,b\} $ and integers $ u(b)\geqslant 1,v(b)\geqslant 0 $ such that
	\begin{equation}\label{eq4}
	  R(x)-R(b)=(x-b)^{u(b)}(x-\alpha(b))^{v(b)}.
	\end{equation}
	For any $ b\in \mathbb{F}_{q}\backslash \{0,a\} $ with $ (k+l)b-ka\neq 0$, since the derivative of $ R(x) $ is
	$$ R'(x)=x^{k-1}(x-a)^{l-1}((k+l)x-ka),$$
	we see $ R'(b)\neq 0 $ and thus from (\ref{eq4}) we have $ u(b)=1,v(b)=k+l-1\geqslant 2$ and $ (k+l)\alpha(b)-ka=0 $. Hence, $k+l$ is not divided by $p$ and $\alpha=\alpha(b)=ka/(k+l)$ is independent of $ b $. Then, for any $ b\in \mathbb{F}_{q}\backslash \{0,a,\alpha\} $, from (\ref{eq4}) we have
	\begin{equation}\label{eq5}
	  x^{k}(x-a)^{l}-b^{k}(b-a)^{l}-(x-b)(x-\alpha)^{k+l-1}=0.
	\end{equation}
	From (\ref{eq1}) we have
	\begin{equation}\label{eq6}
       \lvert \mathbb{F}_{q}\backslash \{0,a,\alpha\}\rvert \geqslant q-3>mn\geq D\geqslant k+l,
	\end{equation}
	and thus from (\ref{eq5}) we see that the $ y$-polynomial
	$$ x^{k}(x-a)^{l}-y^{k}(y-a)^{l}-(x-y)(x-\alpha)^{k+l-1}$$
	is the zero polynomial, which is impossible since its leading term is $y^{k+l}$. Hence, we must have $k=l=1, a\in \mathbb{F}_{q}^{*} $ and thus (\ref{eq3}) follows.
\end{proof}

From Lemmas~\ref{4}, \ref{5} and \ref{6}, we can determine the forms of $f(x)$ and $g(y)$ as showing in the following corollary.
\begin{corollary}
The polynomials $f,g$ can be classified into two cases:
\begin{description}
                   \item[Case 1.] Just one of $f,g$ is injective in $\mathbb{F}_{q^M}$ and thus there are $u\in K_{p}\cap [1,m/2]$, $v\in K_{p}\cap [1,m]$ and $a\in \mathbb{F}_{q}$
                   such that $f(x)=\rho_a^{u}(x), g(y)=y^{v}$ or $f(x)=x^{v}, g(y)=\rho_a^{u}(y)$.
                   \item[Case 2.] Both $f,g$ are injective in $\mathbb{F}_{q^M}$ and thus there are $u,v\in K_{p}\cap [1,m]$ such that $f(x)=x^u,g(y)=x^v$.
\end{description}
\end{corollary}

\section{Characterization of $h$ for Case 1}

In this section, we consider to characterize the polynomial $h(x,y)$ for the first case shown in Corollary~1.

For any $ a\in \mathbb{F}_{q} $, let
$ Q_{{q}}(a) $ denote the set of pairs $ (c,d)\in \mathbb{F}_{q}^{2} $ such that $ \rho_{a}(c),\rho_{a}(d) $ are distinct elements in $ \mathbb{F}_{q}^{*} $, i.e.
\begin{equation}\label{eq00x}
 Q_{{q}}(a)=\{(c,d)\in (\mathbb{F}_{q}\backslash\{0,a\})^{2}:c\neq d,c+d\neq a\}.
\end{equation}
For $j\in[1,n]$, let $h_j\in\mathbb{F}_{q}[x]_{n}$ denote the polynomial defined by
	\begin{align}\label{eq12}
	  h_{j}(x)&=\sum_{1\leqslant i\leqslant n}h_{i,j}x^{i}.
\end{align}
Then, the polynomial $h(x,y)$ can be expressed as $\sum_{1\leq j\leq n}h_j(x)y^j$.
The following lemma gives some conditions on the polynomials $h_j$ for the case $f(x)g(y)=\rho_{a}^{u}(x)y^{v}$, $a\in \mathbb{F}_{q}$, $u,v\in K_{p}$.

\begin{lemma}\label{8}
Suppose $f(x)=\rho_{a}^{u}(x)$, $g(y)=y^{v}$ for some $a\in \mathbb{F}_{q}$
 and $u,v\in K_{p}\cap[1,m]$ with $ 2u\leqslant m$. Let $b$ be the element in $\mathbb{F}_q$ with $b^{v}=a$.
Then, there exists an $ s\in K_{p}\cap [1,n] $ such that
  \begin{gather}\label{eq9}
	  h_{s}(d^{v})\rho_{b}^{su}(c)-h_{s}(c^{v})\rho_{b}^{su}(d)\neq 0,\text{ for any }(c,d)\in Q_{{q^{M}}}(b),\\
\label{eq10}
      h_{s}(c)-h_{s}(a-c)\neq 0,\text{ for any }c\in \mathbb{F}_{q^{M}}\text{ with }2c\neq a,\\
\label{eq11}
	  h(x,y)=h_{s}(x)y^{s}+\sum\limits_{(i,j)\in\Phi_{p}(u,v),j\neq s}h_{2i,j}\rho_{a}^{i}(x)y^{j}.
	\end{gather}
\end{lemma}

\begin{proof} For $ (c,d)\in Q_{q^M}(b) $ and $ t\in\mathbb{F}_{q^{M}}^{*} $, let
	$$ S_{4}=(0,c^{v},d^{v};t\rho_{b}^{u}(d),0,t\rho_{b}^{u}(c)).$$
	Then, from $ v\in K_{p} $ and $ b^{v}=a $ we see
	\begin{equation*}
	\begin{split}
	\Delta_{3}(f(x)g(y))(S_{4})&=f(d^{v})g(t\rho_{b}^{u}(c))-f(c^{v})g(t\rho_{b}^{u}(d))\\
	&=t^{v}(cd)^{uv}\big ((d^{v}-a)^{u}(c-b)^{uv}-( c^{v}-a)^{u}(d-b)^{uv}\big )
	\end{split}
	\end{equation*}
is equal to 0 and thus the $ t $-polynomial
\begin{align}
\Delta_{3}(h(x,y))(S_{4})=&h(d^{v},t\rho_{b}^{u}(c))-h(c^{v},t\rho_{b}^{u}(d))\nonumber\\
=&\sum\limits_{1\leqslant j\leqslant n}(h_{j}(d^{v})\rho_{b}^{ju}(c)-h_{j}(c^{v})\rho_{b}^{ju}(d))t^{j}\nonumber\\
=&\sum\limits_{1\leqslant j\leqslant n}H_{j}(d,c)t^{j}\label{eq17}
\end{align}
has no root in  $ \mathbb{F}_{q^{M}}^{*} $, where
	\begin{equation}\label{eq18}
	  H_{j}(x,y)=h_{j}(x^{v})\rho_{b}^{ju}(y)-h_{j}(y^{v})\rho_{b}^{ju}(x).
	\end{equation}
	Clearly, for any $ j\in [1,n] $, the polynomial $ H_{j}(x,y)$ belongs to $\mathbb{F}_{q}[x,y]_{mn}$. By applying Lemma \ref{7} for $W=Q_{q}(b)\subset Q_{q^M}(b)$, $e_{j}=H_{j}$, $D=n$ and $ N=mn $, according to (\ref{eq1}) we see that there is an integer $ s\in [1,n] $ such that
	\begin{equation}\label{eq19}
	  H_{j}(x,y)=0,\text{ for }j\neq s.
	\end{equation}
	\indent For any $ (c,d)\in Q_{{q^{M}}}(b) $, from (\ref{eq17}) and (\ref{eq19}) we have $ \Delta_{3}(h(x,y))(S_{4})=H_{s}(d,c)t^{s} $ and thus (\ref{eq9}) follows.\\
	\indent If $ j\neq s $ and $ h_{j}(x) $ is not the zero polynomial, from (\ref{eq18}) and (\ref{eq19}) we see
	$$ h_{j}(x^{v})=\frac{h_{j}(c^{v})}{\rho_{b}^{ju}(c)}\rho_{b}^{ju}(x),$$
	and then $ v $ divides $ ju $ and, moreover, from $ v\in K_{p} $ and $ b^{v}=a $ we have
	\begin{equation}\label{eq20}
	  h_{j}(x)=h_{2ju/v,j}\rho_{a}^{ju/v}(x).
	\end{equation}
	Hence, we have
	\begin{equation}\label{eq21}
	  h(x,y)=h_{s}(x)y^{s}+\sum\limits_{j\neq s,v\mid ju}h_{2ju/v,j}\rho_{a}^{ju/v}(x)y^{j}.
	\end{equation}
	\indent For any $ c\in\mathbb{F}_{q^{M}}$ with $2c\neq a$, from $ f(c)=\rho_{a}^{u}(c)=\rho_{a}^{u}(a-c)=f(a-c)$,
	 Lemma \ref{3} and (\ref{eq21}) we see that the $ y $-polynomial
\begin{align*}
	h(c,y)-h(a-c,y)=\sum_{1\leqslant j\leqslant n}(h_{j}(c)-h_{j}(a-c))y^{j}=(h_{s}(c)-h_{s}(a-c))y^{s}
\end{align*}
is injective in $ \mathbb{F}_{q^{M}} $. Therefore, we have (\ref{eq10}) and according to Lemma \ref{6} we see $ s\in K_{p}\cap[1,n] $.

Assume now $c\in\mathbb{F}_{q}\backslash\{0,b\}$. Since for any $j\in [1,n]\setminus K_p$ the $y$-polynomial $\rho_{b}^{j}(c)-\rho_{b}^{j}(y)+(\rho_{b}(y)-\rho_{b}(c))^{j}$
has at most $2j$ roots in $\mathbb{F}_q$, from (\ref{eq1}) we see $q-4>\sum_{2\leq j\leq n}2j=n(n+1)-2$ and thus
there is some $d\in \mathbb{F}_{q}\backslash\{0,b,c,b-c\}$
such that, for any $j\in[1,n]\backslash K_{p}$,
	\begin{equation}\label{eq24}
	  \rho_{b}^{j}(c)-\rho_{b}^{j}(d)+(\rho_{b}(d)-\rho_{b}(c))^{j}\neq 0.
	\end{equation}
Clearly, $(c,d)\in Q_q(b)$. For $ t\in\mathbb{F}_{q^{M}}^{*} $, let
	$$ S_{5}=(0,c^{v},d^{v};0,t\rho_{b}^{u}(d),t(\rho_{b}^{u}(d)-\rho_{b}^{u}(c))).$$
	Then, from $ s\in K_{p} $ we have
	\begin{align}
	&\Delta_{3}(h_{s}(x)y^{s})(S_{5})\nonumber\\
=&(h_{s}(c^{v})-h_{s}(d^{v}))(t\rho_{b}^{u}(d))^{s}+h_{s}(d^{v})(t(\rho_{b}^{u}(d)-\rho_{b}^{u}(c)))^{s}\nonumber\\
=&(h_{s}(c^{v})\rho_{b}^{su}(d)-h_{s}(d^{v})\rho_{b}^{su}(c))t^{s}.\label{eq25}
	\end{align}
	For any $ j\in[1,n] $ with $ v\mid ju $, from $ v\in K_{p} $ and $ b^{v}=a $ we see
	$$ \rho_{b}^{ju}(x)=(x(x-b))^{ju}=(x^{v}(x^{v}-a))^{ju/v}=\rho_{a}^{ju/v}(x^{v})$$
	and thus from $ u\in K_{p} $ we have
	\begin{align}
	&\Delta_{3}(\rho_{a}^{ju/v}(x)y^{j})(S_{5})\nonumber\\
=&\left(\rho_{a}^{ju/v}(c^{v})-\rho_{a}^{ju/v}(d^{v})\right)\left(t\rho_{b}^{u}(d)\right)^{j}
+\rho_{a}^{ju/v}(d^{v})\left(t(\rho_{b}^{u}(d)-\rho_{b}^{u}(c))\right)^{j}
\nonumber\\
=&t^{j}\rho_{b}^{ju}(d)\left(\rho_{b}^{j}(c)-\rho_{b}^{j}(d)+(\rho_{b}(d)-\rho_{b}(c))^{j}\right)^{u}.\label{eq26}
	\end{align}
	\indent If there is some $j\in[1,n]\backslash K_{p}$ with $ v\mid ju $ such that $ h_{2ju/v,j}\neq 0 $, from (\ref{eq9}) and (\ref{eq21}) to (\ref{eq26}) we see that the $ t $-polynomial $ \Delta_{3}(h(x,y))(S_{5})\in \mathbb{F}_{q}[t]_{n} $ has at least two nonzero coefficients, and thus there is some $ t\in\mathbb{F}_{q^{M}}^{*} $ such that $ \Delta_{3}(h(x,y))(S_{5})=0 $. Since $ \Delta_{3}(f(x)g(y))(S_{5})=0 $ can also be obtained from (\ref{eq26}) by replacing $j$ with $v$, the graph $G$ has a 6-cycle of form $ S_{5} $, contradicts to the assumption.\\
	\indent Hence, we have $ h_{2ju/v,j}= 0 $ for any $ j\in[1,n]\backslash K_{p} $ with $ v\mid ju $, and thus (\ref{eq11}) follows from (\ref{eq21}).
\end{proof}

The polynomial $h_s$ in this lemma can be characterized further by the following lemma.
\begin{lemma}\label{9}
	Let $ z(x)=\sum\limits_{1\leqslant i\leqslant n}z_{i}x^{i}\in \mathbb{F}_{q}[x]_{n} $ and
	\begin{equation}\label{eq27}
	  R(x,y)=z(x^{v})\rho_{b}^{w}(y)-z(y^{v})\rho_{b}^{w}(x),
	\end{equation}
	where $ b\in\mathbb{F}_{q}, w\in K_{p}\cap [1,mn/2] $ and $ v\in K_{p}\cap [1,m] $. Suppose
	\begin{equation}\label{eq28}
	  z(c)-z(b^{v}-c)\neq 0,\,\text{for any}\,\, c\in\mathbb{F}_{q}\text{ with }2c\neq b^{v},
	\end{equation}
	\begin{equation}\label{eq29}
	  R(d,c)\neq 0,\,\text{for any}\,\,(c,d)\in Q_{{q^{M}}}(b).
	\end{equation}
Then one of the following three cases is valid.
\begin{enumerate}
  \item $ b\neq 0 $, $ w\geqslant v,z_{w/v}+z_{2w/v}b^{w}\neq 0 $ and
	\begin{equation}\label{eq30}
	  z(x)=(z_{w/v}+z_{2w/v}b^{w})x^{w/v}+z_{2w/v}\rho^{w/v}_b(x).
	\end{equation}
  \item $ b=0 $, $ 2w\geqslant v $ and there are some $ w_{1}\in K_{p} $ and $ \sigma\in\{1,-1\} $ with $ z_{2w/v+\sigma w_{1}}\neq 0 $ such that
	\begin{equation}\label{eq31}
	  z(x)=z_{2w/v+\sigma w_{1}}x^{2w/v+\sigma w_{1}}+z_{2w/v}x^{2w/v}.
	\end{equation}
  \item $z(x)=z_1x$, $z_1\neq 0$ and
  \begin{enumerate}
   \item $b=0$, $p=3$ and $v=3w$, or
   \item $b=0$, $p=2$ and $v=4w$, or
   \item $b\neq 0$, $p=2$ and $v=2w$.
  \end{enumerate}
\end{enumerate}
\end{lemma}

\begin{proof}
	  Let $ [1,n]_{v}=\{v,2v,\dots,nv\} $ and $ R(x,y)=\sum_{i\geqslant 1}r_{i}(y)x^{i} $, where $ r_{i}(y)\in \mathbb{F}_{q}[y] $. From (\ref{eq27}) we have
	
	  \begin{equation}\label{eq33}
	  r_{i}(y)=
	  \begin{cases}
z_{j}(y^{2v}-b^{v}y^{v})^{j}+b^{w}z(y^{v}),&\text{ if } i=w=jv\in[1,n]_{v},\\
b^{w}z(y^{v}),&\text{ if } i=w\notin[1,n]_{v},\\
	  z_{2j}(y^{2v}-b^{v}y^{v})^{j}-z(y^{v}),&\text{ if } i=2w=2jv\in[1,n]_{v},\\
-z(y^{v}),&\text{ if } i=2w\notin[1,n]_{v},\\
	  z_{i/v}\rho_{b}^{w}(y),&\text{ if } i\in[1,n]_{v}\backslash\{w,2w\},\\
0,&\text{ else}.
	  \end{cases}
	  \end{equation}
	
Let $ k $ denote the largest integer with $ r_{k}(y)\neq 0 $ and $ l $ the least integer with $ r_{l}(y)\neq 0 $. Clearly, we have
$1\leqslant l\leqslant k\leqslant mn$
and $R(x,y)=\sum_{l\leq i\leq k}r_i(y)x^i$.
	
For any $c\in\mathbb{F}_{q}\backslash\{0,b\}$ with $2c^v\neq b^v$, from (\ref{eq28}) we have
\begin{align*}
R(b,c)=&z(b^{v})\rho_{b}^{w}(c)-z(c^{v})\rho_{b}^{w}(b)=z(b^{v})\rho_{b}^{w}(c)\neq 0,\text{ if }b\neq 0,\\
R(b-c,c)=&z((b-c)^{v})\rho_{b}^{w}(c)-z(c^{v})\rho_{b}^{w}(b-c)\\
=&(z(b^{v}-c^{v})-z(c^{v}))\rho_{b}^{w}(c)\neq 0,
\end{align*}
and thus, from (\ref{eq29}) and $ R(x,c)\in \mathbb{F}_{q}[x]_{mn} $, we see there are $ \alpha(c)\in\mathbb{F}_{q}^{*} $ and positive integers $l(c),k(c)$ with
$1\leq l(c)<k(c)\leq mn$
such that
\begin{equation}\label{eq35'}
	    R(x,c)=\alpha(c)x^{l(c)}(x-c)^{k(c)-l(c)}.
\end{equation}
	
Let
\begin{align*}
\Theta=\{c\in \mathbb{F}_{q}\backslash\{0,b\}:2c^v\neq b^v, r_{k}(c)\neq 0,r_{l}(c)\neq 0\}.
\end{align*}
For any $ c\in\Theta $, according to the definitions of $k$ and $l$, we have $ l(c)=l $ and $ k(c)=k $, and thus we see
$1\leq l<k\leq mn$ and
\begin{align}\label{eq35}
    R(x,c)=\alpha(c)x^l(x-c)^{k-l}, \text{ for any } c\in\Theta.
\end{align}
For any $ c,d\in\Theta $, from (\ref{eq35}) we see
\begin{align*}%\label{eq35'}
	  \alpha(c)d^{l}(d-c)^{k-l}=R(d,c)=-R(c,d)=-\alpha(d)c^{l}(c-d)^{k-l},
\end{align*}
and then, we have $2\nmid (k-l)$ if $p\neq 2$, and there is some $ \alpha\in\mathbb{F}_{q}^{*} $ such that $ \alpha(c)=\alpha c^{l} $ holds for any $c\in\Theta$. Therefore, we have
\begin{align}\label{eq35''}
	  R(x,c)=\alpha c^{l}x^{l}(x-c)^{k-l},\text{ for any } c\in\Theta.
\end{align}

Since for any $ j $ with $ 1\leqslant 2j\leqslant n $ the polynomial $ z_{2j}(t^{2}-b^{v}t)^{j}-z(t)\in\mathbb{F}_{q}[t] $ has degree at most $ n $ and for any $ j'\in [1,n] $ the polynomial $ z_{j'}(t^{2}-b^{v}t)^{j'}+b^{w}z(t)\in\mathbb{F}_{q}[t] $ has degree at most $ 2n $, from (\ref{eq1}), (\ref{eq33}) and that the polynomial
 $y^v\in\mathbb{F}_q[y]$ is a permutation on $\mathbb{F}_q$, we see easily
\begin{equation}\label{eq36}
	    \lvert\Theta\rvert\geqslant q-1-3n>mn\geqslant k,
\end{equation}
where $ r_{i}(0)=0 $ for any $i$ has been taken into account in the first inequality.
From (\ref{eq35''}) and (\ref{eq36}) we see the polynomial $R(x,y)\in\mathbb{F}_q[x,y]$ is of form
\begin{equation}\label{eq37}
	    R(x,y)=\alpha x^{l}y^{l}(x-y)^{k-l}.
\end{equation}
	
	  \indent Assume $ b\neq 0 $ first. From (\ref{eq27}) and (\ref{eq37}) we have
	  $$\alpha x^{l}b^{l}(x-b)^{k-l}=R(x,b)=-z(b^{v})\rho_{b}^{w}(x)=-z(b^{v})x^{w}(x-b)^{w}$$
	  and thus we see $ z(b^{v})=-\alpha b^{l}, l=k-l=w $ and, for $ c\in \mathbb{F}_{q}\backslash\{0,b\}, $
	  \begin{align}
	  z(x^{v})&=\rho_{b}^{-w}(c)(z(c^{v})\rho_{b}^{w}(x)+\alpha c^{w}x^{w}(x-c)^{w})\nonumber\\
	  &=\rho_{b}^{-w}(c)(z(c^{v})+\alpha c^{w})x^{2w}-\rho_{b}^{-w}(c)(z(c^{v})b^{w}+\alpha c^{2w})x^{w}.\label{eq37'}
	  \end{align}

If $ w\geqslant v $, then from (\ref{eq37'}) we see (\ref{eq30}) and
\begin{align*}
z_{w/v}+z_{2w/v}b^{w}=\alpha\rho_{b}^{-w}(c)c^{w}(b^{w}-c^{w})=-\alpha\neq 0.
\end{align*}

If $ w<v $, then from (\ref{eq37'}) we see  $p=2$, $v=2w$, $z(x)=z_1x$ and $z_1=-\alpha b^{-w}\neq 0$.

	  \indent Assume $ b=0 $ now. For any $ c\in\mathbb{F}_{q}^{*} $, from (\ref{eq27}) and (\ref{eq37}) we have

	  \begin{equation}\label{eq38}
	    z(x^{v})=c^{-2w}(z(c^{v})x^{2w}+\alpha c^{l}x^{l}(x-c)^{k-l}).
	  \end{equation}
Since the left side of (\ref{eq38}) is independent of $ c $, we see that the expansion of $(x-c)^{k-l}$ has at most two terms and thus
we have $k-l\in K_{p}$. Furthermore, from (\ref{eq38}) we have $2w\in\{k,l\}$ and
	  \begin{equation*}
	  z(x^{v})=
	  \begin{cases}
	  c^{-2w}(z(c^{v})+\alpha c^{l})x^{2w}-\alpha x^{l}, &\text{ if }l<2w=k,\\
	  c^{-2w}(z(c^{v})-\alpha c^{k})x^{2w}+\alpha x^{k}, &\text{ if }l=2w<k,
	  \end{cases}
	  \end{equation*}
	  namely, there are integers $ w_{0}\in K_{p} $ and $ \sigma\in\{1,-1\} $ such that
	
	  \begin{equation}\label{eq39}
	    z(x^{v})=c^{-2w}(z(c^{v})-\sigma\alpha c^{2w+\sigma w_{0}})x^{2w}+\sigma\alpha x^{2w+\sigma w_{0}}.
	  \end{equation}
From (\ref{eq39}) and $ \sigma\alpha\neq 0 $ we see
		  \begin{equation}\label{eq40}
	    v\mid (2w+\sigma w_{0}),
	  \end{equation}
$z_{(2w+\sigma w_{0})/v}=\sigma\alpha $ and
	  \begin{equation}\label{eq41}
	  c^{-2w}(z(c^{v})-\sigma\alpha c^{2w+\sigma w_{0}})=
	  \begin{cases}
	  z_{2w/v}, &\text{ if }2w\geqslant v,\\
	  0, &\text{ otherwise}.
	  \end{cases}
	  \end{equation}

If $ 2w\geqslant v $, from (\ref{eq40}) we have $ v\leqslant w_{0} $ and thus from (\ref{eq39}) and (\ref{eq41}) we see that (\ref{eq31}) is true for $ w_{1}=w_{0}/v\in K_{p} $ and $ z_{2w/v+\sigma w_{1}}=\sigma\alpha\neq 0 $.

If $ 2w<v $, from (\ref{eq40}) we have $\sigma=1$,
\begin{align*}
    p=3,\, v=3w,\, w_0=w, \text{ or }p=2,\, v=4w,\, w_0=2w,
\end{align*}
and thus from (\ref{eq39}) and (\ref{eq41}) we see $ z(x)=z_{1}x $ and $ z_{1}=\sigma\alpha\neq 0 $.

The proof is completed.
\end{proof}

It has been shown in Corollary~1 that the polynomials $f,g$ can be classified into two cases, the following theorem characterizes the polynomial $h(x,y)$ further for the first case by using Lemmas~\ref{8} and \ref{9}.

\begin{theorem}\label{t1}
	Let $ a\in\mathbb{F}_{q},u,v\in K_{p} $ with $ u\leqslant m/2 $ and $ v\leqslant m $.\\
	\noindent $ (i) $ If $ f(x)=\rho_{a}^{u}(x) $ and $ g(y)=y^{v} $, then 
there are $ s\in K_{p}\cap [1,n]$ and $ \zeta\in\mathbb{F}_{q}^{*}$ such that either 
\begin{gather}
\label{eq42} v\leq su \text{ and }  \mu_{a,u,v}(h)(x,y)=\zeta x^{su/v}y^{s},\text{ or}\\
\label{eq42''} p=2, a\neq 0, v=2su \text{ and } \mu_{a,u,2su}(h)(x,y)=\zeta xy^{s},\text{ or}\\
\label{eq42'} p=2,a=0, v\leq 4su\text{ and } \mu_{0,u,v}(h)(x,y)=\zeta x^{4su/v}y^{s}.
\end{gather}
	\noindent $ (ii) $ If $ f(x)=x^{v} $ and $ g(y)=\rho_{a}^{u}(y) $, then 
there are $ s\in K_{p}\cap [1,n]$ and $ \zeta\in\mathbb{F}_{q}^{*} $ such that either
\begin{gather}
\label{eq43}
	v\leq su \text{ and }  \nu_{a,u,v}(h)(x,y)=\zeta x^{s}y^{su/v},\text{ or}\\
\label{eq43''}
	p=2, a\neq 0, v=2su \text{ and } \nu_{a,u,2su}(h)(x,y)=\zeta x^{s}y,\text{ or}\\
\label{eq43'}
	 p=2,a=0, v\leq 4su\text{ and } \nu_{0,u,v}(h)(x,y)=\zeta x^{s}y^{4su/v}.
\end{gather}
\end{theorem}

\begin{proof}
	  Since $ (ii) $ is symmetrical to $ (i) $, we only give proof for $ (i) $.\\
	  \indent Assume  $ f(x)=\rho_{a}^{u}(x) $ and $ g(y)=y^{v} $. According to Lemmas~\ref{8} and \ref{9}, we see that there are $ s\in K_{p}\cap [1,n]$ and $ \zeta\in\mathbb{F}_{q}^{*} $ such that either (\ref{eq42}), or (\ref{eq42''}), or	
	  \begin{equation}
\label{eq44}
	    a=0,v\leqslant 2su\text{ and }\mu_{0,u,v}(h)(x,y)=\zeta x^{2su/v+\sigma w_{1}}y^{s}
	  \end{equation}
for some $ w_{1}\in K_{p}\cap [1,n] $ and $ \sigma \in\{1,-1\} $, or
	  \begin{gather}
\label{eq45}
	    a=0,p=3,v=3su\text{ and }\mu_{0,u,3su}(h)(x,y)=\zeta xy^{s}, \text{ or}\\
\label{eq45'}
	    a=0,p=2,v=4su\text{ and }\mu_{0,u,4su}(h)(x,y)=\zeta xy^{s}.
	  \end{gather}

Assume that (\ref{eq44}) is valid for some $ w_{1}\in K_{p}\cap [1,n] $ and $ \sigma \in\{1,-1\} $. From Lemma \ref{iso} we see that $ G $ is isomorphic to
	  $$ G_{1}=\Gamma_{{q^{M}}}(x^{2su/v+\sigma w_{1}}y^{s},x^{2u}y^{v}),$$
	  and thus $G_1$ has girth at least 8.
If $\sigma=-1$, from Lemmas \ref{5} and \ref{6} we have $w_1=su/v$ and that (\ref{eq42}) is valid for $ a=0 $. If $\sigma=1$,
from Lemmas \ref{5} and \ref{6} we have either $p=2, w_1=2su/v$ or $p=3, w_1=su/v$.
Clearly, the former case implies (\ref{eq42'}), and the later case is impossible since $ G_{1} $ is isomorphic to
	  $\Gamma_{{q^{M}}}(x^{3}y,x^{2}y) $
	  whose girth is 6 according to ($ii$) of Lemma~\ref{2}.

	  \indent Assume that (\ref{eq45}) is valid. Then, from Lemma \ref{iso} and $ p=3 $ we see that $ G $ is isomorphic to
	  $$ \Gamma_{{q^{M}}}(x^{2u}y^{3su},xy^{s})\cong \Gamma_{{q^{M}}}(x^{2}y^{3},xy)$$
	  whose girth is 6 according to ($ii$) of Lemma~\ref{2}, contradicts to that $ G $ has girth at least 8.

	  \indent Clearly, (\ref{eq45'}) implies (\ref{eq42'}) for $v=4su$.
\end{proof}

\section{Characterization of $h$ for Case 2}

The following theorem characterizes the polynomial $ h(x,y) $ for the second case of Corollary~1.

\begin{theorem}\label{t2}
	Assume $ f(x)=x^{u} $ and $ g(y)=y^{v} $ for some $ u,v\in K_{p}\cap[1,m]. $
	Then, either there is some $ s\in K_{p}\cap[1,n] $ with $ v\leqslant su $ and $ h_{2su/v,s}\neq 0 $ such that
	\begin{equation}\label{eq46}
	  \pi_{u,v}(h)h(x,y)=h_{2su/v,s}x^{2su/v}y^{s},
	\end{equation}
	or there is some $ r\in K_{p}\cap[1,n] $ with $ u\leqslant rv $ and $ h_{r,2rv/u}\neq 0 $ such that
	\begin{equation}\label{eq47}
	  \pi_{u,v}(h)(x,y)=h_{r,2rv/u}x^{r}y^{2rv/u}.
	\end{equation}
\end{theorem}

\begin{proof}
	  For $ a,b\in\mathbb{F}_{q^{M}}^{*},t\in\mathbb{F}_{q^{M}}\backslash\{0,1\} $, let
	  \begin{equation}\label{eq48}
	    S_{6}=(0,at^{v},a;0,b,b(1-t)^{u}).
	  \end{equation}
	
	  Then, from $ \Delta_{3}(f(x)g(y))(S_{6})=\Delta_{3}(x^{u}y^{v})(S_{6})=0 $ we see that, for any $ a\in\mathbb{F}_{q^{M}}^{*} $ and $ t\in\mathbb{F}_{q^{M}}\backslash\{0,1\} $, the $ b $-polynomial
	  \begin{equation}\label{eq49}
	    \Delta_{3}(h(x,y))(S_{6})=\sum\limits_{1\leqslant j\leqslant n}E_{j}(a,t)b^{j}
	  \end{equation}
has no root in $ \mathbb{F}_{q^{M}}^{*} $, where
	  \begin{equation}\label{eq50}
	    E_{j}(x,y)=h_{j}(xy^{v})-h_{j}(x)+h_{j}(x)(1-y)^{uj}.
	  \end{equation}
	
	  Since $E_{j}$ is a polynomial in $\mathbb{F}_{q}[x,y]_{mn}$, by applying Lemma \ref{7} for $ W=\mathbb{F}_{q}^{*}\times(\mathbb{F}_{q}\backslash\{0,1\}), D=n, N=mn $ and $e_{j}=E_{j}$, according to (\ref{eq1}) we see that there is an $ s\in[1,n] $ such that
	  \begin{equation}\label{eq51}
	    E_{j}(x,y)=0,\text{ for any } j\neq s.
	  \end{equation}
Therefore, we have $ \Delta_{3}(h(x,y))(S_{6})=E_{s}(a,t)b^{s} $ and thus
	  \begin{equation}\label{eq52}
	    E_{s}(a,t)\neq 0, \text{ for any } a\in \mathbb{F}_{q^{M}}^{*} \text{ and } t\in \mathbb{F}_{q^{M}}\backslash\{0,1\}.
	  \end{equation}
	
	  \indent For $ j\notin K_{p}\cup\{s\} $, from (\ref{eq50}) and (\ref{eq51}) we see
\begin{equation*}
	  h_{j}(xy^{v})=h_{j}(x)(1-(1-y)^{uj})=h_j(x)(1-(1-y)^j)^u
\end{equation*}
and thus, for any positive integer $k$,
	  \begin{equation*}
	  h_{j}(x)(1-(1-y^{k})^{j})^u=h_{j}(xy^{kv})=h_{j}(x)(1-(1-y)^{j})^{ku}.
	  \end{equation*}
Hence, we have
	  \begin{equation}\label{eq53}
	    h_{j}(x)=0,\text{ if } j\notin K_{p}\cup\{s\}.
	  \end{equation}
	
	  \indent For $ j\in K_{p}\backslash\{s\} $, from (\ref{eq50}) and (\ref{eq51}) we see $ h_{j}(xy^{v})=h_{j}(x)y^{ju} $ and thus we have $ v\leqslant ju $ and
	  \begin{equation}\label{eq54}
h_{j}(x)=\left\{
\begin{array}{ll}
	  h_{ju/v,j}x^{ju/v}, &\text{ if } v\leqslant ju,\\
	  0,&\text{ otherwise}.
\end{array}
\right.
	  \end{equation}
	  \indent By following the proof of Lemma \ref{7}, from (\ref{eq52}) and
	  $$ E_{s}(t,y)=\sum\limits_{1\leqslant i\leqslant n}h_{i,s}(y^{iv}-1+(1-y)^{su})t^{i},$$
	  one can show that there is an $ r\in[1,n] $ such that $ h_{r,s}\neq 0 $ and
	  \begin{equation}\label{eq55}
	    t^{rv}-1+(1-t)^{su}\neq 0, \text{ for any } t\in \mathbb{F}_{q^{M}}\backslash\{0,1\},
	  \end{equation}
	  \begin{equation}\label{eq56}
	    h_{i,s}(y^{iv}-1+(1-y)^{su})=0,\text{ for any } i\neq r.
	  \end{equation}
	  From (\ref{eq55}), we see $ su\neq rv $ if $ s\in K_{p} $. Therefore, from (\ref{eq56}) we see
	  \begin{align*}
	  h_{s}(x)=\left\{
	  \begin{array}{ll}
	  h_{r,s}x^{r}+h_{su/v,s}x^{su/v}, &\text{ if } s\in K_{p}\text{ and } v\leqslant su,\\
	  h_{r,s}x^{r},&\text{ otherwise},
	  \end{array}\right.
	  \end{align*}
	  and thus from (\ref{eq53}) and (\ref{eq54}) we have
	  \begin{equation}\label{eq57}
	    \pi_{u,v}(h)(x,y)=h_{r,s}x^{r}y^{s}.
	  \end{equation}
	  According to (\ref{eq57}) and Lemma \ref{iso}, we see that $ G_{2}=\Gamma_{{q^{M}}}(x^{r}y^{s},x^{u}y^{v}) $ is isomorphic to $ G $ and thus has girth at least 8.\\
	  \indent Assume $ r,s\in K_{p} $ first. From $ su\neq rv $ we see there is a $ w_{2}\in K_{p}\backslash\{1\} $ such that either $su=rvw_{2}$ or $rv=suw_{2}$. If the polynomial $t^{w_2-1}-1\in\mathbb{F}_q[t]$ has some root $t_0$ in $\mathbb{F}_{q^M}\setminus\{1\}$, then $G_{2}$ contains the 6-cycle $(0,1,t_{0}^{v};t_{0}^{u},0,1)$, contradicts to that $G_2$ has girth at least 8. Hence, $t^{w_2-1}-1\in\mathbb{F}_q[t]$ has no root in $\mathbb{F}_{q^M}\setminus\{1\}$ and thus $w_2=2=p$, and either (\ref{eq46}) or (\ref{eq47}) is valid.

	  \indent Assume $ r\not\in K_{p} $ now. According to Lemmas \ref{5} and \ref{6}, we have $s\in K_{p}$, $2\mid r$, $r/2 \in K_{p}$ and $p\neq 2$. Then, from (\ref{eq55}) we see $t^{rv}-t^{su}\neq 0$ for any $t\in \mathbb{F}_{q^{M}}\backslash\{0,1\}$, and thus
	  \begin{equation}\label{eq58}
	    \lvert rv-su\rvert\in K_{p}.
	  \end{equation}
	  If $ rv<su $, then from (\ref{eq58}) we have $ p=3,\,3rv=2su $ and thus from Lemma~\ref{iso} we see that $ G_{2} $ is isomorphic to
	  $$ \Gamma_{{q^{M}}}(x^{3vr}y^{3vs},x^{u}y^{v})\cong \Gamma_{{q^{M}}}(x^{2su}y^{3vs},x^{u}y^{v})\cong \Gamma_{{q^{M}}}(x^{2}y^{3},xy)$$
	  whose girth is of 6, contradicts to that $ G_{2} $ has girth at least 8. Hence, we have $ rv>su $ and thus from (\ref{eq58}) we see $ rv=2su $ and (\ref{eq46}).\\
	  \indent Similarly, one can show that (\ref{eq47}) is valid if $s\not\in K_{p}$.
\end{proof}

\section{Proof of the main result and concluding remarks}

\indent In this section, we complete the proof of the main result of this paper and give some concluding remarks.
\medskip

\noindent
{\it Proof of Theorem~\ref{one}}:
One hand, we assume that the bipartite graph $ G=\Gamma_{{q^{M}}}(f(x)g(y),h(x,y)) $ has girth of at least eight.
According to Corollary~1, Theorems~\ref{t1} and \ref{t2}, we see that the polynomials $f,g,h$ must be of the desired forms,
namely, there are some $ a\in \mathbb{F}_{q} $, $ \zeta\in\mathbb{F}_{q}^{*} $, $ u,v\in K_{p}\cap[1,m] $ and $s\in K_{p}\cap[1,n]$ such that one of ($i$) to ($iv$) is valid,
where it should be noticed that (\ref{eq42'}) and (\ref{eq43'}) are included in ($iii$) indeed.

On the other hand, we assume that there are some $ a\in \mathbb{F}_{q} $, $ \zeta\in\mathbb{F}_{q}^{*} $, $ u,v\in K_{p}\cap[1,m] $ and $s\in K_{p}\cap[1,n]$ such that one of ($i$) to ($iv$) is valid. According to the isomorphisms given in Lemma~\ref{iso}, one can show easily that
the graph $G$ must be isomorphic to $\Gamma_3(\mathbb{F}_{q^M})$. For example, if ($a$) of ($iv$) is valid, then from $u,s\in K_p$, $p=2$ and $a\in\mathbb{F}^*_q$ we see $G$ is isomorphic to
\begin{align*}
&\Gamma_{q^M}(\rho_a^u(x)y^{2su},xy^s)\cong \Gamma_{q^M}(\rho_a(x)y^{2s},xy^s)\cong\Gamma_{q^M}(\rho_a(x)y^{2},xy)\\
\cong& \Gamma_{q^M}(\rho_a(x)y^{2},x^2y^2)\cong\Gamma_{q^M}(xy^{2},x^2y^2)\cong \Gamma_{q^M}(xy^{2},xy)\cong\Gamma_3(\mathbb{F}_{q^M}).
\end{align*}

The proof is completed.
\hspace*{\fill}{$\square$}
\medskip

Clearly, according to Theorem~\ref{one} and Lemma~\ref{iso}, one can characterize easily all of the polynomials $f,g\in\mathbb{F}_q[x]_m$ and $h\in\mathbb{F}_q[x,y]_n$ for which the graphs $\Gamma_{q^M}(f(x)g(y),h(x,y))$ have girth at least eight, where $M=\text{lcm}\{2,3,\ldots,mn\}$. In particular, for any assemble of such polynomials $f,g,h$ and any positive integer $k$, the graph $\Gamma_{q^{k}}(f(x)g(y),h(x,y))$ is isomorphic to $\Gamma_3(\mathbb{F}_{q^{k}})$.

The integer $M$ is chosen in such a way so as to ensure any polynomial in $\mathbb{F}_{q}[x]_{mn}$ is decomposable in $\mathbb{F}_{q^M}$.
However, only the decomposability in $\mathbb{F}_{q^M}$ of the polynomials in $\mathbb{F}_q[x]_n$ is demanded excepting the proof of Lemma~\ref{9} needs such property for the polynomials of form $R(x,c)=z(x^v)\rho_a^w(c)-z(c^v)\rho_a^w(x)\in \mathbb{F}_{q}[x]_{mn}$, where $z\in\mathbb{F}_q[x]_n$, $a\in\mathbb{F}_q$, $v\in K_p\cap[1,m]$ and $w\in K_p\cap[1,mn/2]$.
Hence, $M$ may be replaced with a smaller integer if alternative proofs of Lemma~\ref{9} are available.

Even if the condition (\ref{eq1}) is not valid, the conclusions of Theorem~\ref{one} are still true provided the integer $M$ is replaced by $rM$ for any positive integer $r$ with
\begin{align*}%\label{eq1}
q^r > \max\{2mn+3,mn+3n+1,n(n+1)+2\},
\end{align*}
namely, it is sufficient to replace the basic field $\mathbb{F}_q$ with $\mathbb{F}_{q^r}$ in the proof.

At the end of this paper, in stead of giving a detailed proof for Theorem~\ref{two}, we just point out that it can be proved by simply following the clues of the proof of Theorem~\ref{one}. In fact, a proof of Theorem~\ref{two} can be obtained directly from that of Theorem~\ref{one} if we replace the prime powers $q, q^M$ by $\infty$, the finite fields $\mathbb{F}_q$, $\mathbb{F}_{q^M}$ by $\mathbb{F}_{\infty}$,
the set $K_{p}$ by $\{1\}$ (the set $\Phi_p(u,v)$ is then equal to $\{(1,1)\}$), respectively, and the main arguments are still true while some of them can be simplified by skipping the redundant discussions. For example, in the proof of Theorem~\ref{one}, almost the restrictions on the degree of polynomials over $\mathbb{F}_q$ are designed to ensure their decomposability in the extension field $\mathbb{F}_{q^M}$, and thus these restrictions can be simply dropped since $\mathbb{F}_{\infty}$ is an algebraically closed field.

\renewcommand\refname{Reference}


\begin{thebibliography}{10}
	
 \bibitem{Bollobas98} B. Bollob\'{a}s, Modern Graph Theory, Springer, New York, 1998.	

 \bibitem{Hou17} X. Hou, S.D. Lappano, F. Lazebnik, Proof of a Conjecture on Monomial Graphs, Finite Fields Appl. 43 (2017) 42-68.

\bibitem{Bondy74} J.A. Bondy, M. Simonovits, Cycles of even length in graphs, J. Comb. Theory, Ser. B 16 (1974) 97-105.

\bibitem{Lazebnik95} F. Lazebnik, V.A. Ustimenko, A.J. Woldar, A new series of dense graphs of high girth, Bull. Am. Math. Soc. (N.S.) 32 (1995) 73-79.

 \bibitem{Dmytrenko04} V. Dmytrenko, Classes of polynomial graphs, PhD thesis, University of Delaware, 2004.

 \bibitem{Dmytrenko07} V. Dmytrenko, F. Lazebnik, J. Williford, On monomial graphs of girth eight, Finite Fields Appl. 13 (2007) 828-842.

 \bibitem{Kronenthal12} B.G. Kronenthal, Monomial graphs and generalized quadrangles, Finite Fields Appl. 18 (2012) 674-684.

 \bibitem{Kronenthal16} B.G. Kronenthal, F. Lazebnik, On the uniqueness of some girth eight algebraically defined graphs, Discrete Appl. Math. 206 (2016) 188-194.

 \bibitem{Kronenthal19} B.G. Kronenthal, F. Lazebnik, J. Williford, On the uniqueness of some girth eight algebraically defined graphs, Part II, Discrete Appl. Math. 254 (2019) 161-170.

\end{thebibliography}
\end{document}